\def\DateTime{3/April/2009, 10:00 (Kyoto)}
\def\Version{Version $1.0$}
\def\yes{\if00}
\def\no{\if01}
\def\iftwelvept{\no}

\def\ifusepdf{\no}
\def\ifpsfont{\yes}

\iftwelvept
\documentclass[leqno,12pt]{amsart}
\else
\documentclass[leqno]{amsart}
\fi
\usepackage{amssymb}
\usepackage{amscd}
\usepackage{latexsym}
\usepackage{verbatim}
\usepackage[all]{xy}

\ifusepdf
\usepackage{hyperref}
\else\fi
\ifpsfont
\usepackage[T1]{fontenc}
\usepackage{times}
\else\fi

\iftwelvept
\else\fi

\theoremstyle{plain}
\newtheorem{Theorem}{Theorem}[section]
\newtheorem{Proposition}[Theorem]{Proposition}
\newtheorem{Lemma}[Theorem]{Lemma}
\newtheorem{Corollary}[Theorem]{Corollary}

\newtheorem{Claim}{Claim}[Theorem]

\theoremstyle{definition}

\newtheorem{Remark}[Theorem]{Remark}
\newtheorem{Example}[Theorem]{Example}

\def\rom{\textup}
\newcommand{\ZZ}{{\mathbb{Z}}}
\newcommand{\QQ}{{\mathbb{Q}}}
\newcommand{\RR}{{\mathbb{R}}}
\newcommand{\CC}{{\mathbb{C}}}
\newcommand{\PP}{{\mathbb{P}}}
\newcommand{\KK}{{\mathbb{K}}}

\newcommand{\OO}{{\mathcal{O}}}
\newcommand{\Proj}{\operatorname{Proj}}

\newcommand{\Image}{\operatorname{Image}}

\newcommand{\Coker}{\operatorname{Coker}}
\newcommand{\Spec}{\operatorname{Spec}}

\newcommand{\rank}{\operatorname{rk}}

\newcommand{\adeg}{\widehat{\operatorname{deg}}}
\newcommand{\avol}{\widehat{\operatorname{vol}}}

\newcommand{\zeros}{\operatorname{div}}

\newcommand{\quot}{\operatorname{quot}}

\newcommand{\Bs}{\operatorname{Bs}}

\newcommand{\CQED}{{\unskip\nobreak\hfil\penalty50\quad\null\nobreak\hfil
{$\Box$}\parfillskip0pt\finalhyphendemerits0\par\medskip}}
\newcommand{\rest}[2]{\left.{#1}\right\vert_{{#2}}}

\newcommand{\Nm}{\operatorname{\|\text{\textperiodcentered}\|}}

\begin{document}

\title[Free basis consisting of strictly small sections]%
{Free basis consisting of strictly small sections}
\author{Atsushi Moriwaki}
\address{Department of Mathematics, Faculty of Science,
Kyoto University, Kyoto, 606-8502, Japan}
\email{moriwaki@math.kyoto-u.ac.jp}
\date{\DateTime, (\Version)}
\begin{abstract}
Let $X$ be a projective arithmetic variety and $\overline{L}$ a continuous
hermitian invertible sheaf on $X$.
In this paper, we would like to consider a mild and appropriate condition to
guarantee that $H^0(X, nL)$ has a free basis consisting of strictly small sections
for $n \gg 1$.
\end{abstract}

\maketitle

\renewcommand{\theTheorem}{\Alph{Theorem}}

\section*{Introduction}

Let $X$ be a $d$-dimensional projective arithmetic variety and $L$ an
invertible sheaf on $X$. We fix a continuous hermitian metric $\vert\cdot\vert$ of $L$.
In Arakelov geometry, we frequently ask whether
$H^0(X,L)$ has a free basis consisting of strictly small sections,
that is, sections whose supremum norm are less than $1$.
However, we know few about this problem in general.
For example, Zhang \cite{ZhPos} proves that
$H^0(X, nL)$ possesses a free basis as above for $n \gg 1$ if
the following conditions are satisfied:
\begin{enumerate}
\renewcommand{\labelenumi}{\rom{(\arabic{enumi})}}
\item
$X_{\QQ}$ is regular, $L_{\QQ}$ is ample and
$L$ is nef on every fiber of $X \to \Spec(\ZZ)$.

\item
The metric $\vert\cdot\vert$ is $C^{\infty}$ and
the first Chern form $c_1(L,\vert\cdot\vert)$ is semipositive.

\item
For some positive integer $n_0$,
there are strictly small sections $s_1, \ldots, s_l$ of $n_0L$ such that
$\{ x \in X_{\QQ} \mid s_1(x) = \cdots = s_l(x) = 0 \} = \emptyset$.
\end{enumerate}
From viewpoint of birational geometry, the ampleness of $L_{\QQ}$ is
rather strong. Through Example~\ref{example:projective:line} and 
Example~\ref{example:projective:plane}, the base point freeness is not a necessarily
condition, but we can realize that
the condition (3) is substantially crucial to find a basis consisting of strictly small sections.
In this paper, we would like to consider the problem under the 
mild and appropriate assumption (3) (cf. Corollary~\ref{cor:intro:B}).

Let $R$ be a graded subring of $\bigoplus_{n=0}^{\infty} H^0(X, nL)$ over $\ZZ$.
For each $n$, we assign a norm $\Nm_n$ to $R_n \otimes_{\ZZ} \RR$ in such a way that
\[
\Vert s \cdot s' \Vert_{n+n'} \leq \Vert s \Vert_{n} \cdot \Vert s' \Vert_{n'}
\]
holds for all $s \in R_n \otimes_{\ZZ} \RR$ and $s' \in R_{n'} \otimes_{\ZZ} \RR$.
Then 
\[
(R, \Nm) = \bigoplus_{n=0}^{\infty} (R_n, \Nm_n)
\]
is called a {\em normed graded subring of $L$}. An important point is that
each norm $\Nm_n$ does not necessarily arise from the metric of $L$, so that
we can obtain several advantages to proceed with arguments.
The following theorem is one of the main results of this paper.

\begin{Theorem}
\label{thm:intro:A}
If $R \otimes_{\ZZ} \QQ$ is noetherian and
there are homogeneous elements $s_1, \ldots, s_l \in R$ of positive degree
such that 
$\{ x \in X_{\QQ} \mid s_1(x) = \cdots = s_l(x) = 0 \} = \emptyset$,
then there is a positive constant $B$ such that
\[
\lambda_{\ZZ}(R_n,\Nm_n) \leq B n^{(d+2)(d-1)/2}
\left(  \max \left\{ \Vert s_1 \Vert^{1/\deg(s_1)},
\ldots, \Vert s_l \Vert^{1/\deg(s_l)} \right\}\right)^n
\]
for all $n \geq 1$,
where $\lambda_{\ZZ}(R_n,\Nm_n)$ is the infimum of the set of real numbers $\lambda$
such that
there is a free $\ZZ$-basis $x_1, \ldots, x_r$ of  $R_n$ with
$\max \{ \Vert x_1 \Vert_n, \ldots, \Vert x_r \Vert_n \} \leq \lambda$.
\end{Theorem}

This is a consequence of the technical result Theorem~\ref{thm:base:strictly:small:sec},
which also yields variants of arithmetic Nakai-Moishezon's criterion
(cf. Theorem~\ref{thm:Nakai:Moishezon:1} and Theorem~\ref{thm:Nakai:Moishezon:2}).
As a corollary of the above theorem, we have the following:

\begin{Corollary}
\label{cor:intro:B}
Let $\overline{L}$ be a continuous hermitian invertible sheaf on $X$.
If the above condition \rom{(3)} is satisfied, in other words,
\[
\left\langle \{ s \in H^0(X, n_0L) \mid \Vert s \Vert_{\sup} < 1 \} \right\rangle_{\ZZ} \otimes_{\ZZ} \OO_X
\to n_0 L
\]
is surjective on $X_{\QQ}$ for some positive integer $n_0$,
then
$H^0(X, nL)$ has a free $\ZZ$-basis consisting of strictly small sections for $n \geq 1$.
\end{Corollary}

\renewcommand{\theTheorem}{\arabic{section}.\arabic{Theorem}}
\renewcommand{\theClaim}{\arabic{section}.\arabic{Theorem}.\arabic{Claim}}
\renewcommand{\theequation}{\arabic{section}.\arabic{Theorem}.\arabic{Claim}}

\section{Normed $\ZZ$-module}
Let $(V,\Nm)$ be a normed finite dimensional vector space over $\RR$, that is,
$V$ is a finite dimensional vector space over $\RR$ and $\Nm$ is a norm of $V$.
Let $\alpha : W \to V$ be an injective homomorphism of finite dimensional vector spaces 
over $\RR$.
If we set $\Vert w \Vert_{W \hookrightarrow V} = \Vert\alpha(w)\Vert$ for $w \in W$, then
$\Nm_{W \hookrightarrow V}$ gives rise to a norm of $W$. This is called the {\em subnorm} of
$W$ induced by $W \hookrightarrow V$ and the norm $\Nm$ of $V$.
Next let $\beta : V \to T$ be a surjective homomorphism of finite dimensional vector spaces over $\RR$.
The {\em quotient norm} $\Nm_{V \twoheadrightarrow T}$ of $T$ induced by $V \twoheadrightarrow T$ and the norm $\Nm$ of $V$
is given by 
\[
\Vert t \Vert_{V \twoheadrightarrow T} = \inf \{ \Vert v \Vert \mid \beta(v) = t \}
\]
for $t \in T$.
Let $(U,\Nm)$ be another normed finite dimensional vector space over $\RR$, and let
$\phi : V \to U$ be a homomorphism over $\RR$. The norm $\Vert \phi\Vert$ of $\phi$ is defined to be
\[
\Vert \phi \Vert = \sup \{ \Vert \phi(v) \Vert \mid v \in V, \ \Vert v \Vert = 1 \}.
\]
First let us see the following lemma.

\begin{Lemma}
\label{norm:sub:quot:4:spaces}
Let $(V,\Nm)$ be a normed finite dimensional vector space over $\RR$.
Let $T \subseteq U \subseteq W \subseteq V$ be vector subspaces of $V$.
Then
\[
(\Nm_{W \hookrightarrow V})_{W \twoheadrightarrow W/U} = ((\Nm_{V \twoheadrightarrow V/T})_{W/T \hookrightarrow V/T})_{W/T \twoheadrightarrow W/U}
\]
holds on $W/U$.
\end{Lemma}

\begin{proof}
Let us consider the following commutative diagram:
\[
\xymatrix{
W \ar@{^{(}->}[r] \ar@{->>}[d] & V \ar@{->>}[d] \\
W/U \ar@{^{(}->}[r] & V/U.
}
\]
Then, by \cite[(2) in Lemma~3.4]{MoCont}, we have
\[
(\Nm_{W \hookrightarrow V})_{W \twoheadrightarrow W/U} = (\Nm_{V \twoheadrightarrow V/U})_{W/U \hookrightarrow V/U}.
\]
Moreover, considering the following commutative diagram:
\[
\xymatrix{
W/T \ar@{^{(}->}[r] \ar@{->>}[d] & V/T \ar@{->>}[d] \\
W/U \ar@{^{(}->}[r] & V/U,
}
\]
if we set $\Nm' = \Nm_{V \twoheadrightarrow V/T}$, then
\[
(\Nm'_{V/T \twoheadrightarrow V/U})_{W/U \hookrightarrow V/U} =
(\Nm'_{W/T \hookrightarrow V/T})_{W/T \twoheadrightarrow W/U}.
\]
Thus the lemma follows because
$\Nm_{V \twoheadrightarrow V/U} = \Nm'_{V/T \twoheadrightarrow V/U}$
by  \cite[(1) in Lemma~3.4]{MoCont}.
\end{proof}

Let $M$ be a finitely generated $\ZZ$-module and $\Nm$ a norm of $M_{\RR} := M \otimes_{\ZZ} \RR$.
A pair $(M, \Nm)$ is called a {\em normed $\ZZ$-module}.
For a normed $\ZZ$-module $(M, \Nm)$, 
we define $\lambda_{\QQ}(M, \Nm)$ and $\lambda_{\ZZ}(M,\Nm)$ to be
\[
\hspace{-2.7em}
\lambda_{\QQ}(M,\Nm)  := \inf \left\{ \lambda \in \RR  \left|
\begin{array}{l}
\text{there are $e_1, \ldots, e_n \in M$ such that $e_1, \ldots, e_n $} \\
\text{form a basis over $\QQ$ and $\Vert e_i \Vert \leq \lambda$ for all $i$} 
\end{array}
\right\}\right. 
\]
and
\[
\lambda_{\ZZ}(M,\Nm) := \inf \left\{ \lambda \in \RR  \left|
\begin{array}{l}
\text{there are $e_1, \ldots, e_n \in M$ such that $e_1, \ldots, e_n$ form} \\
\text{a free $\ZZ$-basis of $M/M_{tor}$ and $\Vert e_i \Vert \leq \lambda$ for all $i$} 
\end{array}
\right\}\right..
\]
Note that if $M$ is a torsion module, then $\lambda_{\QQ}(M, \Nm) = \lambda_{\ZZ}(M,\Nm) = 0$.

\begin{Lemma}
\label{lem:lambda:lambda:prime}
$\lambda_{\QQ}(M,\Nm) \leq \lambda_{\ZZ}(M,\Nm) \leq \rank (M) \lambda_{\QQ}(M,\Nm)$.
\end{Lemma}

\begin{proof}
See \cite[Lemma~1.7 and its consequence]{ZhPS}.
\end{proof}

\begin{Lemma}
\label{lem:iso:Q:comp:lambda}
Let $(M_1,\Nm_1)$ and $(M_2, \Nm_2)$ be normed $\ZZ$-modules, and let
$\phi : M_1 \to M_2$ be a homomorphism such that $\phi$ yields an isomorphism over $\QQ$.
Then we have the following:
\begin{enumerate}
\renewcommand{\labelenumi}{\rom{(\arabic{enumi})}}
\item
$\lambda_{\QQ}(M_2,\Nm_2) \leq \Vert \phi \Vert \lambda_{\QQ}(M_1, \Nm_1)$.

\item
Further we assume that $\phi$ is surjective and that $\phi$ induces an isometry 
\[
((M_{1})_{\RR}, \Nm_1) \overset{\sim}{\longrightarrow} ((M_{2})_{\RR}, \Nm_2).
\]
Then $\lambda_{\QQ}(M_2,\Nm_2) = \lambda_{\QQ}(M_1, \Nm_1)$.
\end{enumerate}
\end{Lemma}

\begin{proof}
(1) Let $e_1, \ldots, e_n \in M_1$ such that
$e_1, \ldots, e_n$ form a basis of $M_1$ over $\QQ$ and
\[
\max \{ \Vert e_1 \Vert_1, \ldots, \Vert e_n \Vert_1 \} = \lambda_{\QQ}(M_1, \Nm_1).
\]
Then  $\phi(e_1), \ldots, \phi(e_n)$ form a basis of $M_2$ over $\QQ$ and
\[
\Vert \phi(e_i) \Vert_2 \leq \Vert \phi \Vert \Vert e_i \Vert_1 \leq  \Vert \phi \Vert \lambda_{\QQ}(M_1, \Nm_1)
\]
for all $i$. Thus we have the assertion.

(2) First of all, by (1),
$\lambda_{\QQ}(M_2,\Nm_2) \leq \lambda_{\QQ}(M_1, \Nm_1)$.
Let $y_1, \ldots, y_n \in M_2$ such that
$y_1, \ldots, y_n$ form a basis of $M_2$ over $\QQ$ and
\[
\max \{ \Vert y_1 \Vert_2, \ldots, \Vert y_n \Vert_2 \} = \lambda_{\QQ}(M_2, \Nm_2).
\]
For each $i$, we choose $x_i \in M_1$ with $\phi(x_i) = y_i$.
Then $\Vert x_i \Vert_1 = \Vert y_i \Vert_2$ for all $i$.
Thus $\lambda_{\QQ}(M_1,\Nm_1) \leq \lambda_{\QQ}(M_2, \Nm_2)$.
\end{proof}

\begin{Proposition}
\label{prop:estimate:lambda:chain}
Let $(M_1,\Nm_1), \ldots, (M_n, \Nm_n)$ be normed $\ZZ$-modules.
For each $i$ with $2 \leq i \leq n$, let $\alpha_i : M_{i-1} \to M_i$ be a homomorphism such that
$\alpha_i$ gives rise to an injective homomorphism over $\QQ$.
We set \
\[
\phi_i = \alpha_n \circ \cdots \circ \alpha_{i+1} : M_i \to M_n
\]
for $i=i,\ldots,n-1$, and
\[
Q_i = \begin{cases}
\Coker(\alpha_i : M_{i-1} \to M_i) & \text{if $i \geq 2$}, \\
M_1 & \text{if $i=1$}
\end{cases}
\]
for $i=1,\ldots,n$. Then we have
\[
\lambda_{\QQ}(M_n, \Nm_n) \leq
\lambda_{\QQ}(Q_n, \Nm_{n, M_n \twoheadrightarrow Q_n}) + \sum_{i=1}^{n-1} \Vert \phi_i \Vert
\lambda_{\QQ}(Q_i, \Nm_{i, M_i \twoheadrightarrow Q_i}) \rank Q_i.
\]
\end{Proposition}

\begin{proof}
The proof of this proposition can be found in \cite[Lemma~5.1]{ZhPS}.
For reader's convenience, we reprove it here.

Let $\Nm'_i = \Nm_{n,(M_i)_{\RR} \hookrightarrow (M_n)_{\RR}}$, that is,
the sub-norm induced by the injective homomorphism
$\phi_i : (M_i)_{\RR} \to (M_n)_{\RR}$ and the norm $\Nm_n$ of $(M_n)_{\RR}$.
First let us see the following claim.

\begin{Claim}
\label{claim:prop:estimate:lambda:chain:1}
$\lambda_{\QQ}(Q_i, \Nm'_{i, M_i  \twoheadrightarrow Q_i})
\leq \Vert  \phi_i \Vert
\lambda_{\QQ}(Q_i, \Nm_{i, M_i  \twoheadrightarrow Q_i})$.
\end{Claim}

By the definition of
$\Vert  \phi_i \Vert$,
for $x \in (M_i)_{\RR}$,
\[
\Vert x \Vert_i \Vert  \phi_i \Vert \geq \Vert \phi_i(x) \Vert = \Vert x \Vert'_i.
\]
Thus, for $y \in (Q_i)_{\RR}$,
\[
\Vert y \Vert_{i,M_i  \twoheadrightarrow Q_i} \Vert  \phi_i \Vert 
\geq \Vert y \Vert'_{i,M_i  \twoheadrightarrow Q_i},
\]
which shows the inequality of the claim.
\CQED

By Claim~\ref{claim:prop:estimate:lambda:chain:1},
(2) in Lemma~\ref{lem:iso:Q:comp:lambda} and replacing
$M_i$ with $\phi_i(M_i)$, we may assume that
$\alpha_i : M_{i-1} \hookrightarrow M_i$ is an inclusion map and
$\Nm_i = \Nm_{n, M_i \hookrightarrow M_n}$.

\begin{Claim}
\label{claim:prop:estimate:lambda:chain:2}
The assertion holds in the case $n=2$, that is,
\[
\lambda_{\QQ}(M_2,\Nm_2) \leq \lambda_{\QQ}(Q_2,\Nm_{2,M_2 \twoheadrightarrow Q_2}) + \lambda_{\QQ}(M_1,\Nm_1) \rank M_1.
\]
\end{Claim}

Let $e_1, \ldots, e_s \in M_1$ and $f_1, \ldots, f_t \in Q_2$ such that
$e_1, \ldots, e_s$ and $f_1, \ldots, f_t$ form bases  of $M_1$ and $Q_2$ over $\QQ$ respectively, and that
\[
\begin{cases}
\lambda_{\QQ}(M_1,\Nm_1) = \max \{ \Vert e_1 \Vert_1, \ldots, \Vert e_s \Vert_1 \}, \\
\lambda_{\QQ}(Q_2,\Nm_{2,M_2 \twoheadrightarrow Q_2}) =
\max \{ \Vert f_{1}\Vert_{2,M_2 \twoheadrightarrow Q_2}, \ldots, \Vert f_t \Vert_{2,M_2 \twoheadrightarrow Q_2} \}.
\end{cases}
\]
Let us choose $f'_j \in M_2$ and $f''_j \in (M_2)_{\RR}$ such that
$f'_j = f_j$ on $Q_2$, $f''_j = f_j$ on $(Q_2)_{\RR}$ and that $\Vert f''_j \Vert_2 = \Vert f_{j}\Vert_{2,M_2 \twoheadrightarrow Q_2}$.
Since $f'_j \otimes 1 - f''_j \in (M_1)_{\RR}$,
there are $a_{ji} \in \RR$ such that
\[
f'_j \otimes 1 - f''_j  = \sum_{i} a_{ji} (e_i \otimes 1).
\]
We set $g_j = f'_j - \sum_i \lfloor a_{ji} \rfloor e_i$.
Then $e_1, \ldots, e_s, g_1, \ldots, g_t \in M_2$
form a basis of $M_2$ over $\QQ$.
Moreover, as 
\[
g_j \otimes 1 = f''_j + \sum_i (a_{ji} - \lfloor a_{ji} \rfloor)(e_i \otimes 1),
\]
we have
\[
\Vert g_j \Vert_2 \leq \lambda_{\QQ}(Q_2,\Nm_{2,M_2 \twoheadrightarrow Q_2}) + \lambda_{\QQ}(M_1,\Nm_1) \rank M_1 ,
\]
which implies the claim.
\CQED

We assume $n \geq 3$.
We set $M'_i = M_{i}/M_1$ for $i=2, \ldots, n$ and the norm $\Nm'_i$ of $M'_i$ is given by
$\Nm'_i = \Nm_{i, M_{i} \twoheadrightarrow M'_i}$.
Note that 
\[
\Nm'_i = (\Nm_{n, M_n \twoheadrightarrow M'_{n}})_{M'_{i} \hookrightarrow M'_{n}}
\]
by
\cite[(2) in Lemma~3.4]{MoCont}.
Applying the induction hypothesis to
\[
(M'_2, \Nm'_2) \hookrightarrow \cdots  \hookrightarrow (M'_{n}, \Nm'_{n}),
\]
we obtain
\[
\lambda_{\QQ}(M'_{n}, \Nm'_{n}) \leq
\lambda_{\QQ}(Q_n, \Nm'_{n, M'_n \twoheadrightarrow Q_n}) + \sum_{i=2}^{n-1}
\lambda_{\QQ}(Q_i, \Nm'_{i, M'_i \twoheadrightarrow Q_i}) \rank Q_i.
\]
Using Lemma~\ref{norm:sub:quot:4:spaces} in the case where 
\[
M_1 \subseteq M_{i-1} \subseteq M_i \subseteq M_n,
\]
we have $\Nm'_{i, M'_i \twoheadrightarrow Q_i} = \Nm_{i, M_i \twoheadrightarrow Q_i}$.
Therefore, the above inequality means
\addtocounter{Claim}{1}
\begin{equation}
\label{eqn:prop:estimate:lambda:chain:1}
\lambda_{\QQ}(M'_{n}, \Nm'_{n}) \leq
\lambda_{\QQ}(Q_n, \Nm_{n, M_n \twoheadrightarrow Q_n}) + \sum_{i=2}^{n-1}
\lambda_{\QQ}(Q_i, \Nm_{i, M_i \twoheadrightarrow Q_i}) \rank Q_i.
\end{equation}
On the other hand, applying Claim~\ref{claim:prop:estimate:lambda:chain:2} 
to the case where $(M_1, \Nm_1) \hookrightarrow (M_n, \Nm_n)$,
we can see
\addtocounter{Claim}{1}
\begin{equation}
\label{eqn:prop:estimate:lambda:chain:2}
\lambda_{\QQ}(M_n,\Nm_n) \leq \lambda_{\QQ}(M'_n,\Nm'_{n}) + \lambda_{\QQ}(M_1,\Nm_1) \rank M_1,
\end{equation}
so that we obtain the assertion combing \eqref{eqn:prop:estimate:lambda:chain:1} with
\eqref{eqn:prop:estimate:lambda:chain:2}.
\end{proof}

\section{Normed graded ring}
Let $k$ be a commutative ring with unity and 
$R = \bigoplus_{n=0}^{\infty}  R_n$ a graded ring over $k$.
Let $M$ be a $R$-module and $h$ a positive integer.
We say $M$ is a {\em $h$-graded $R$-module} if
$M$ has a decomposition $M = \bigoplus_{n=-\infty}^{\infty} M_n$ as $k$-modules and
\[
x \in R_n,\ m \in M_{n'} \quad\Longrightarrow\quad x \cdot m \in M_{hn + n'}
\]
holds for all $n \in \ZZ_{\geq 0}$ and $n' \in \ZZ$.
For example, if we set $R^{(h)} = \bigoplus_{n=0}^{\infty} R_{nh}$, then
$R$ is a $h$-graded $R^{(h)}$-module.
Form now on, we assume that $k = \ZZ$ and
$R_n$ (resp. $M_n$) is a finitely generated $\ZZ$-module 
for all $n \in \ZZ_{\geq 0}$ (resp. $n \in \ZZ$).
Let $\KK$ be either $\QQ$ or $\RR$.
We set $R_{\KK} = R \otimes_{\ZZ} \KK$ and $M_{\KK} = M \otimes_{\ZZ} \KK$. Then
\[
R_{\KK} = \bigoplus_{n=0}^{\infty}  (R_n)_{\KK}\quad\text{and}\quad
M_{\KK} = \bigoplus_{n=-\infty}^{\infty}  (M_n)_{\KK},
\]
where $(R_n)_{\KK} = R_n \otimes_{\ZZ} \KK$ and
$(M_n)_{\KK} = M_n \otimes_{\ZZ} \KK$.
Note that $R_{\KK}$ is a graded ring over $\KK$ and
$M_{\KK}$ is a $h$-graded $R_{\KK}$-module.
We say 
\[
(R,\Nm) = \bigoplus_{n=0}^{\infty}  (R_n,\Nm_n)
\]
is a {\em normed graded ring over $\ZZ$}
if 
\begin{enumerate}
\renewcommand{\labelenumi}{\rom{(\arabic{enumi})}}
\item
$\Nm_n$ is a norm of $(R_{n})_{\RR}$ for each $n \in \ZZ_{\geq 0}$, and
\item
$\Vert s \cdot s' \Vert_{n+n'} \leq \Vert s \Vert_{n} \Vert s' \Vert_{n'}$ holds for all $s \in (R_{n})_{\RR}$ and $s' \in (R_{n'})_{\RR}$.
\end{enumerate}
Similarly, 
\[
(M,\Nm_M) = \bigoplus_{n=-\infty}^{\infty}  (M_n,\Nm_{M_n})
\]
is called a {\em normed $h$-graded $(R,\Nm)$-module}
if 
\begin{enumerate}
\renewcommand{\labelenumi}{\rom{(\arabic{enumi})'}}
\item
$\Nm_{M_n}$ is a norm of $(M_{n})_{\RR}$ for each $n \in \ZZ$, and
\item
$\Vert s \cdot m \Vert_{M_{hn+n'}} \leq \Vert s \Vert_{n} \Vert m \Vert_{M_{n'}}$ holds for all $s \in (R_{n})_{\RR}$ and $m \in (M_{n'})_{\RR}$.
\end{enumerate}

\begin{Proposition}
\label{prop:normed:graded:structure:descent}
Let $I$ be a homogeneous ideal of $R$ and $R' = R/I$.
Let $f : M \to Q$ be a surjective homomorphism of $h$-graded $R$-modules of degree $0$, that is,
$f(M_n) = Q_n$ for all $n \in \ZZ$.
We set 
\[
(R',\Nm') = \bigoplus_{n=0}^{\infty}  (R'_n,\Nm'_{n})\quad\text{and}\quad
(Q,\Nm_Q) = \bigoplus_{n=-\infty}^{\infty}  (Q_n,\Nm_{Q_n}),
\]
where
$\Nm'_n = \Nm_{n,R_n \twoheadrightarrow R'_n}$ and
$\Nm_{Q_n} = \Nm_{M_n, M_n \twoheadrightarrow Q_n}$.
Then we have the following:
\begin{enumerate}
\renewcommand{\labelenumi}{\rom{(\arabic{enumi})}}
\item
$(R',\Nm')$ is a normed graded ring over $\ZZ$.

\item
If $I \cdot Q = 0$, then $(Q,\Nm_Q)$ is naturally a normed $h$-graded $(R',\Nm')$-module.
\end{enumerate}
\end{Proposition}

\begin{proof}
(1) We need to see that
\[
\Vert x' \cdot y' \Vert'_{n+n'} \leq \Vert x' \Vert'_n \Vert y' \Vert'_{n'}
\]
for all $x' \in (R'_{n})_{\RR}$ and $y' \in (R'_{n'})_{\RR}$. 
Indeed, we choose $x \in (R_{n})_{\RR}$ and $y \in (R_{n'})_{\RR}$ such that the classes of
$x$ and  $y$ in $R'_{\RR}$ are $x'$ and $y'$ respectively and that
$\Vert x \Vert_n = \Vert x' \Vert'_{n}$ and
$\Vert y \Vert_n = \Vert y' \Vert'_{n'}$.
Then, as  the class of $x \cdot y$ in $R'_{\RR}$ is $x' \cdot y'$,
\[
\Vert x' \cdot y' \Vert'_{n+n'} \leq \Vert x \cdot y \Vert_{n+n'} \leq
\Vert x \Vert_{n} \Vert y \Vert_{n'} =  \Vert x' \Vert'_{n} \Vert y' \Vert'_{n'}.
\]

(2) It is sufficient to show that
\[
\Vert x' \cdot q \Vert_{Q_{hn+n'}} \leq \Vert x' \Vert'_n \Vert q \Vert_{Q_{n'}}
\]
for all $x' \in (R'_{n})_{\RR}$ and $q \in (Q_{n'})_{\RR}$, which can be checked in the same way as in (1).
\end{proof}

Next let us observe the following lemma:

\begin{Lemma}
\label{lem:asym:R:asym:M}
We assume the following:
\begin{enumerate}
\renewcommand{\labelenumi}{\rom{(\arabic{enumi})}}
\item
$M_{\QQ}$ is a finitely generated $R_{\QQ}$-module, and
$M_n = \{ 0 \}$ for $n < 0$.

\item
There are $A,e,\upsilon \in \RR_{> 0}$ 
such that
$\lambda_{\QQ}(R_n,\Nm_n) \leq A n^e \upsilon^n$
for all $n \geq 1$.
\end{enumerate}
Then there is $A' \in \RR_{> 0}$ such that
$\lambda_{\QQ}(M_n,\Nm_{M_n}) \leq A' n^e \upsilon^{n/h}$
for all $n \geq 1$.
\end{Lemma}

\begin{proof}
For $n \geq 1$,
we choose $s_{n,1}, \ldots, s_{n,r_n} \in R_n$ such that
$s_{n,1}, \ldots, s_{n,r_n}$ form a basis of $(R_{n})_{\QQ}$ and
$\Vert s_{n,j} \Vert_n \leq A n^e \upsilon^n$ holds for all $j=1,\ldots, r_n$.
Let $m_1, \ldots, m_l$ be homogeneous elements of $M_{\QQ}$ such that
$M_{\QQ}$ is generated by $m_1, \ldots, m_l$ as a $R_{\QQ}$-module.
Let $a_i$ be the degree of $m_i$.
Clearly we may assume that $m_i \in M_{a_i}$ by replacing $m_i$ with $b m_i$ ($b \in \ZZ_{>0}$).
If $n > \max \{ a_1, \ldots, a_l \}$, then $(M_{n})_{\QQ}$ is generated by elements of the form $s_{i,j} m_k$ with $i h + a_k = n$ 
and $i \geq 1$.
We set 
\[
B = \max_{k=1, \ldots, l} \left\{ \frac{\Vert m_k \Vert_{M_{a_k}} \upsilon^{-a_k/h}}{h^e} \right\}.
\]
Note that $s_{i,j} m_k \in M_n$ and
\begin{align*}
\Vert s_{i,j} m_k \Vert_{M_n}  & \leq \Vert s_{i,j} \Vert_i \Vert m_k \Vert_{M_{a_k}} \leq A i^e \upsilon^i \Vert m_k \Vert_{M_{a_k}} \\
& =
A \left( \frac{n - a_k}{h} \right)^e \upsilon^{(n-a_k)/h}\Vert m_k \Vert_{M_{a_k}} \leq AB n^e \upsilon^{n/h} 
\end{align*}
which means that 
$\lambda_{\QQ}(M_n, \Nm_{M_n}) \leq AB n^e \upsilon^{n/h}$ holds
for all $n >    \max \{ a_1, \ldots, a_l \}$, as required.
\end{proof}

As a consequence, we have the following proposition.

\begin{Proposition}
\label{prop:I:J:K:estimate}
Let $I$, $J$ and $K$ be homogeneous ideals of $R$ such that $J \subseteq K$
and $I \cdot K \subseteq J$.
We set $R' = R/I$ as before and $Q = K/J$.
Let $\Nm_{K_n} = \Nm_{n,K_n \hookrightarrow R_n}$ and
$\Nm_{Q_n} = \Nm_{K_n, K_n \twoheadrightarrow Q_n}$.
If $R_{\QQ}$ is noetherian and
there are $A, e, \upsilon \in \RR_{> 0}$ 
such that
\[
\lambda_{\QQ}(R'_n,\Nm'_n) \leq A n^e \upsilon^n
\]
for all $n \geq 1$,
then there is $A' \in \RR_{> 0}$ such that
\[
\lambda_{\QQ}(Q_n,\Nm_{Q_n}) \leq A' n^e \upsilon^{n}\]
for all $n \geq 1$.
\end{Proposition}

\begin{proof}
Obviously, $(K, \Nm_K) = \bigoplus_{n=0}^{\infty} (K_n, \Nm_{K_n})$ 
is a normed $1$-graded $(R, \Nm)$-module.
Thus, by Proposition~\ref{prop:normed:graded:structure:descent}, 
$(Q, \Nm_Q) =\bigoplus_{n=0}^{\infty} (Q_n, \Nm_{Q_n})$ 
is also a normed $1$-graded $(R, \Nm)$-module.
As $I \cdot Q = 0$, by Proposition~\ref{prop:normed:graded:structure:descent} again,
$(Q, \Nm_Q)$ is a normed $1$-graded $(R', \Nm')$-module.
Since $R_{\QQ}$ is noetherian and $K_{\QQ}$ is an ideal of $R_{\QQ}$,
$K_{\QQ}$ is finitely generated as a $R_{\QQ}$-module.
Thus $Q_{\QQ}$ is also finitely generated as a $R'_{\QQ}$-module.
Hence the assertion follows from Lemma~\ref{lem:asym:R:asym:M}.
\end{proof}

Finally note the following lemma, which will be used later.

\begin{Lemma}
\label{lem:R:finite:over:R:m}
Let $R = \bigoplus_{n=0}^{\infty} R_n$ be a graded ring and $h$ a positive integer. 
If $R$ is noetherian, then $R^{(h)}$ is also noetherian and
$R$ is a finitely generated $R^{(h)}$-module.
\end{Lemma}

\begin{proof}
See \cite[Chap.~III, \S~1, $\text{n}^{\circ}$~3, Proposition~2 and its proof]{Bourbaki}.
\end{proof}

\section{Estimation of $\lambda_{\QQ}$ for a normed graded ring}

Let $X$ be a $d$-dimensional projective arithmetic variety, that is,
$X$ is a $d$-dimensional projective and flat integral scheme over $\ZZ$,
and let $L$ be an invertible sheaf on $X$.
Let $R$ be a graded subring of $\bigoplus_{n=0}^{\infty} H^0(X, nL)$ over $\ZZ$.
Such a graded ring $R$ is called a {\em graded subring of $L$}.
For each $n$, we assign a norm $\Nm_n$ to $(R_{n})_{\RR}$ such that
$(R,\Nm) = \bigoplus_{n=0}^{\infty} (R_n,\Nm_n)$ is a normed graded ring over $\ZZ$.

For an ideal sheaf  $\mathcal{I}$  of $X$, we set
\[
\begin{cases}
I_n(R;\mathcal{I}) = H^0(X, nL \otimes \mathcal{I}) \cap R_n,\\
I(R;\mathcal{I}) = \bigoplus_{n=0}^{\infty} I_n(R;\mathcal{I}),\\
R_{\mathcal{I}} = R/I(R;\mathcal{I}).
\end{cases}
\]
Then $I(R;\mathcal{I})$ is a homogeneous ideal of $R$.
Let $\Nm_{(R_\mathcal{I})_n}$ be the quotient norm of $(R_{\mathcal{I}})_n$ induced
by $R_n \twoheadrightarrow (R_{\mathcal{I}})_{n}$ and the norm $\Nm_n$ of $R_n$.
Let $Y$ be an arithmetic subvariety of $X$, that is,
$Y$ is an integral closed 
subscheme flat over $\ZZ$, and $\mathcal{I}_Y$ the defining ideal sheaf of $Y$.
Then, for simplicity, 
$R_{\mathcal{I}_Y}$,  $\Nm_{R_{\mathcal{I}_Y}}$, $(R_Y)_n$,  and $\Nm_{(R_Y)_n}$ 
are denoted by $R_Y$, $\Nm_Y$, $R_{Y,n}$ and $\Nm_{Y,n}$ respectively.
Note that 
\[
\xymatrix{
R_{Y,n} \ar@{^{(}->}[r] & H^0(X,nL)/H^0(X, nL \otimes \mathcal{I}_Y) \ar@{^{(}->}[r] & H^0(Y, \rest{nL}{Y}).
}
\]
Thus $R_Y$ is a graded subring of $\rest{L}{Y}$ and 
\[
R_{Y,n} \overset{\sim}{\longrightarrow} \Image(R_n \to H^0(Y, \rest{nL}{Y})).
\]
In particular, $R_Y$ is an integral domain.
We denote the set of all arithmetic subvarieties of $X$ by $\Sigma_X$.
The following theorem is the technical main theorem of this paper.

\begin{Theorem}
\label{thm:base:strictly:small:sec}
Let
$\upsilon : \Sigma_X \to \RR_{> 0}$ be a map.
For $(R, \Nm)$ and $\upsilon$, 
we assume the following:
\begin{enumerate}
\renewcommand{\labelenumi}{\rom{(\arabic{enumi})}}
\item
$R_{\QQ}$ is noetherian.

\item
For each $Y \in \Sigma_X$, 
there is $n_0 \in \ZZ_{>0}$ such that
$(R_{Y,n})_{\QQ} = H^0\left(Y_{\QQ}, \rest{n L_{\QQ}}{Y_{\QQ}}\right)$ for all $n \geq n_0$.

\item
For each $Y \in \Sigma_X$, 
there 
are $n_1 \in \ZZ_{>0}$ and
$s \in R_{Y,n_1} \setminus \{ 0 \}$ with $\Vert s \Vert_{Y, n_{1}}  \leq \upsilon(Y)^{n_1}$.
\end{enumerate}
Then there are $B \in \RR_{>0}$ 
and  a finite subset $S$ of $\Sigma_X$
such that
\[
\lambda_{\QQ}(R_n, \Nm_n) \leq B n^{d(d-1)/2} \left(\max\{ \upsilon(Y) \mid Y \in S \}\right)^n
\]
for all $n \geq 1$.
\end{Theorem}

\begin{proof}
This theorem can be proved by similar techniques as in \cite[Theorem~(4.2)]{ZhPos}.
Let $D \in \Sigma_X$ and $\upsilon_D = \rest{\upsilon}{\Sigma_D}$, where
$\Sigma_{D}$ is the set of all arithmetic subvarieties of $D$.
Note that the conditions (1), (2) and (3) also hold for $(R_D, \Nm_D)$ and $\upsilon_D$.
Let us begin with the following claim.

\begin{Claim}
\label{claim:thm:base:strictly:small:sec:1}
We may assume that there is a non-zero $s \in R_1$ with $\Vert s \Vert_1 \leq \upsilon(X)$.
\end{Claim}

We choose a positive integer $m$ and a non-zero section $s \in R_m$ with
$\Vert s \Vert_m \leq \upsilon(X)^m$.
Clearly the assumptions (1) and (2) of the theorem hold for
$R^{(m)} = \bigoplus_{n=0}^{\infty} R_{mn}$.
For $Y \in \Sigma_X$,
we choose a positive integer $n_1$ and a non-zero $t \in R_{Y, n_1}$
with $\Vert t \Vert_{Y, n_1} \leq \upsilon(Y)^{n_1}$.
Then $t^m \in R_{Y, mn_1} \setminus \{ 0\}$ and
\[
\Vert t^m \Vert_{Y,mn_1} \leq (\Vert t \Vert_{Y,n_1})^m \leq (\upsilon(Y)^m)^{n_1}.
\]
Thus  $(R^{(m)}, \Nm^{(m)})$ and $\upsilon^m$ satisfy
the assumption (3) of the theorem.
Therefore, if the theorem holds for $(R^{(m)},\Nm^{(m)})$ and $\upsilon^m$, then
there are $B \in \RR_{> 0}$ and a finite subset $S$ of $\Sigma_X$
such that
\[
\lambda_{\QQ}(R_{nm}, \Nm_{nm}) \leq B n^{d(d-1)/2} \left(\max\{ \upsilon(Y)^m \mid Y \in S \}\right)^{n}
\]
for all $n \geq 1$.
On the other hand, 
by Lemma~\ref{lem:R:finite:over:R:m}, $R_{\QQ}$ is a finitely generated 
$R^{(m)}_{\QQ}$-module. Thus, by  Lemma~\ref{lem:asym:R:asym:M}, there is 
$B' \in \RR_{> 0}$ such that
\[
\lambda_{\QQ}(R_{n}, \Nm_{n}) \leq B' n^{d(d-1)/2} \left(\max\{ \upsilon(Y)^m \mid Y \in S\}\right)^{n/m}
\]
for all $n \geq 1$. Therefore the claim follows.
\CQED

\begin{Claim}
\label{claim:thm:base:strictly:small:sec:2}
The assertion of the theorem holds if $d=1$.
\end{Claim}

Since $R_n \overset{\cdot s}{\longrightarrow} R_{n+1}$ is injective,
\[
\rank R_1 \leq \cdots \leq \rank R_n \leq \rank R_{n+1} \leq \cdots \leq \rank L.
\]
Thus there is a positive integer $n_0$ such that $R_{n_0} \overset{\cdot s^{n}}{\longrightarrow} R_{n_0 + n}$
yields an isomorphism over $\QQ$. Hence, by (1) in Lemma~\ref{lem:iso:Q:comp:lambda},
\[
\lambda_{\QQ}(R_{n+n_0}, \Nm_{n+n_0}) \leq \Vert s \Vert_1^n \lambda_{\QQ}(R_{n_0}, \Nm_{n_0}) \leq \upsilon(X)^n \lambda_{\QQ}(R_{n_0}, \Nm_{n_0}),
\]
as required. 
\CQED

We prove the theorem on induction of 
$d$. By Claim~\ref{claim:thm:base:strictly:small:sec:2}, 
we have done in the case $d=1$.
Thus we assume $d > 1$.
Let $\mathcal{I}$ be the ideal sheaf of $\OO_X$ given by
\[
\mathcal{I} = \Image \left( L^{-1} \overset{ \otimes s}{\longrightarrow} \OO_X \right).
\]

\begin{Claim}
\label{claim:thm:base:strictly:small:sec:3}
There is a sequence
\[
\mathcal{I}_0 = \mathcal{I} \subsetneq  \mathcal{I}_1 \subsetneq \cdots \subsetneq  \mathcal{I}_m = \OO_X
\]
of ideal sheaves and proper integral subschemes 
$D_1, \ldots, D_m$ of $X$ such that $ \mathcal{I}_{D_r} \cdot  \mathcal{I}_{r} \subseteq  \mathcal{I}_{r-1}$ for all $r=1,\ldots,m$,
where  $\mathcal{I}_{D_r}$ is the defining ideal sheaf of $D_r$.
\end{Claim}

It is standard. For example, we can show it by using
\cite[Chapter~1, Proposition~7.4]{Hartshorne}.
\CQED

Let us fix a positive integer $n_1$ such that
$(R_{n})_{\QQ} = H^0(X_{\QQ}, nL_{\QQ})$ for all $n \geq n_1$.
We set 
\[
\overline{R}_n = (R_n, \Nm_n)\quad\text{and}\quad
\overline{I}_n(R;\mathcal{I}_r) = (I_n(R;\mathcal{I}_r), \Nm_{n,r}),
\]
where  $\Nm_{n,r} = \Nm_{n, I_n(R;\mathcal{I}_r) \hookrightarrow R_n}$.
Note that $\overline{R}_n = \overline{I}_n(R;\mathcal{I}_m)$.
We would like to apply Proposition~\ref{prop:estimate:lambda:chain} to
\addtocounter{Claim}{1}
\begin{equation}
\label{eqn:thm:base:strictly:small:sec:1}
\begin{array}{ccccccc}
\overline{R}_{n_1}  & \overset{\cdot s}{\longrightarrow} & \overline{I}_{n_1+1}(R; \mathcal{I}_{0}) & \hookrightarrow \cdots \hookrightarrow &
\overline{I}_{n_1+1}(R; \mathcal{I}_{r}) & \hookrightarrow \cdots \hookrightarrow & \overline{I}_{n_1+1}(R;\mathcal{I}_m) \\
& \overset{\cdot s}{\longrightarrow} & \overline{I}_{n_1+2}(R; \mathcal{I}_{0}) & \hookrightarrow \cdots \hookrightarrow &
\overline{I}_{n_1+2}(R; \mathcal{I}_{r}) & \hookrightarrow \cdots \hookrightarrow & \overline{I}_{n_1+2}(R;\mathcal{I}_m) \\
 & \vdots &  \vdots & \vdots &  \vdots & \vdots & \vdots \\
& \overset{\cdot s}{\longrightarrow} & \overline{I}_{n}(R; \mathcal{I}_{0}) & \hookrightarrow \cdots \hookrightarrow &
\overline{I}_{n}(R; \mathcal{I}_{r}) & \hookrightarrow \cdots \hookrightarrow & \overline{I}_{n}(R;\mathcal{I}_m).
\end{array}
\end{equation}
For this purpose, let us observe the following claim.

\begin{Claim}
\label{claim:thm:base:strictly:small:sec:4}
\begin{enumerate}
\renewcommand{\labelenumi}{\rom{(\alph{enumi})}}
\item
Let $\Nm_{n,r,\quot}$ be the quotient norm of $I_n(R;\mathcal{I}_r)/I_n(R;\mathcal{I}_{r-1})$ induced by
$I_n(R;\mathcal{I}_r) \twoheadrightarrow I_n(R;\mathcal{I}_r)/I_n(R;\mathcal{I}_{r-1})$
and $\Nm_{n,r}$ of
$I_n(R;\mathcal{I}_r)$. Then,
for each $1 \leq r \leq m$, there are $B_r \in \RR_{> 0}$ and a finite subset $S_r$ of $\Sigma_X$ such that 
\begin{multline*}
\hspace{3em}
\lambda_{\QQ}\left( I_n(R;\mathcal{I}_r)/I_n(R;\mathcal{I}_{r-1}), \
\Nm_{n, r, \quot}
\right) \\
\leq B_r n^{(d-1)(d-2)/2}  \left(\max\{ \upsilon(Y) \mid Y \in S_r \}\right)^{n}.
\end{multline*}
for all $n \geq 1$.

\item
If we set 
\[
e_{n,r} = \max \{ 1, \rank(I_n(R;\mathcal{I}_r)/I_n(R;\mathcal{I}_{r-1})) \},
\]
then there is $C_1 \in \RR_{>0}$ such that $e_{n,r} \leq C_1 n^{d-2}$
for all $n \geq 1$ and $r=1, \ldots, m$.

\item
$\rank (I_n(R;\mathcal{I}_0)/R_{n-1} s) = 0$ for all $n \geq n_1 + 1$.
\end{enumerate}
\end{Claim}

(a) If $D_r$ is vertical, then $I_n(R;\mathcal{I}_r)/I_n(R;\mathcal{I}_{r-1})$
is a torsion module for all $n \geq 0$. 
Thus the assertion is obvious. In this case, we can set $S_r = \{ X \}$ and $B_r = 1$.
Otherwise, 
since $I(R;\mathcal{I}_{D_r}) \cdot I(R;\mathcal{I}_r) \subseteq I(R;\mathcal{I}_{r-1})$,
the assertion follows from 
Proposition~\ref{prop:I:J:K:estimate} and
the hypothesis of induction.

(b) Note that $I_n(R;\mathcal{I}_r)/I_n(R;\mathcal{I}_{r-1}) \hookrightarrow H^0(D_r, nL \otimes \mathcal{I}_r/\mathcal{I}_{r-1})$.

(c) It follows from
\[
(R_{n-1})_{\QQ} s = H^0(X_{\QQ}, (n-1)L_{\QQ}) s = H^0(X_{\QQ}, (nL \otimes \mathcal{I})_{\QQ}) =
I_n(R;\mathcal{I})_{\QQ}.
\]
\CQED

Using (c) in Claim~\ref{claim:thm:base:strictly:small:sec:4} and applying Proposition~\ref{prop:estimate:lambda:chain}
to \eqref{eqn:thm:base:strictly:small:sec:1},
we obtain
\begin{multline*}
\lambda_{\QQ}(R_n,\Nm_n) \\
\leq
\sum_{i=n_1 + 1}^n  \left(  \sum_{r=1}^m \Vert s \Vert_1^{n-i}
\lambda_{\QQ}\left( I_i(R;\mathcal{I}_r)/I_i(R;\mathcal{I}_{r-1}), \
\Nm_{i, r, \quot} \right) e_{i,r}
\right) \\
+
\Vert s \Vert_1^{n-n_1} \lambda(R_{n_1}, \Nm_{n_1}) \rank(R_{n_1})
\end{multline*}
for $n \geq n_1 + 1$.
Hence, if we set 
$S = S_1 \cup \cdots \cup S_r \cup \{ X \}$, 
then, using (a) and (b) in Claim~\ref{claim:thm:base:strictly:small:sec:4},
the theorem follows.
\end{proof}

For homogeneous elements $s_1, \ldots, s_l$ of $R$, we define $\Bs_{\QQ}(s_1, \ldots, s_l)$ to be
\[
\Bs_{\QQ}(s_1, \ldots, s_l) = \{ x \in X_{\QQ} \mid s_1(x) = \cdots = s_l(x) = 0 \}.
\]
As an application of Theorem~\ref{thm:base:strictly:small:sec}, we have the following theorem.

\begin{Theorem}
\label{thm:base:free:noetherian:lambda}
If $R_{\QQ}$ is noetherian and
there are homogeneous elements $s_1, \ldots, s_l \in R$ of positive degree
such that  $\Bs_{\QQ}(s_1, \ldots, s_l) = \emptyset$,
then there is a positive constant $B$ such that
\[
\lambda_{\QQ}(R_n, \Nm_n) \leq B n^{d(d-1)/2}\left( \max \left\{  \Vert s_1 \Vert^{1/\deg(s_1)},\ldots, 
\Vert s_l \Vert^{1/\deg(s_l)}\right\}\right)^n,
\]
for all $n \geq 1$.
\end{Theorem}

\begin{proof}
Let us begin with the following claim:
\begin{Claim}
\label{claim:thm:base:free:noetherian:lambda:1}
We may assume that $R$ is generated by $R_1$ over $R_0$ and that $s_1, \ldots, s_l \in R_1$.
\end{Claim}

Since $R_{\QQ}$ is noetherian, there are homogeneous elements $x_1, \ldots, x_r \in R_{\QQ}$
such that $R_{\QQ} = (R_{0})_{\QQ}[x_1, \ldots, x_r]$ 
(cf. \cite[Chap.~III, \S~1, $\text{n}^{\circ}$~2, Corollaire]{Bourbaki}).
Replacing $x_i$ with $m x_i$ ($m \in \ZZ_{> 0}$), we may assume that $x_i \in R$ for all $i$.
We set 
\[
R' = R_0[x_1, \ldots, x_r, s_1, \ldots, s_l]
\]
in $R$. Then $R'_{\QQ} = R_{\QQ}$.
As $R_{n}/R'_n$ is a torsion module, by (1) in Lemma~\ref{lem:iso:Q:comp:lambda},
we have 
$\lambda_{\QQ}(R_n, \Nm_n) \leq \lambda_{\QQ}(R'_n,\Nm_n)$ for all $n \geq 0$.
Thus we may assume that $R$ is noetherian.
Therefore, there is a positive integer $h$ such that $R^{(h)}$ 
is generated by $R_h$ over $R_0$ 
(cf. \cite[Chap.~III, \S~1, $\text{n}^{\circ}$~3, Proposition~3]{Bourbaki}).
Letting $a_i$ be the degree of $s_i$, we set $a = a_1 \cdots a_l$ and $s'_i = s_i^{ha_1 \cdots a_{i-1} a_{i+1} \cdots a_l}$ for each $i$.
Then $s'_1, \ldots, s'_l \in R_{ah}$ and 
\[
\max \{ \Vert s'_1 \Vert, \ldots, \Vert s'_l \Vert\} \leq \left( \max \left\{  \Vert s_1 \Vert^{1/\deg(s_1)},\ldots, 
\Vert s_l \Vert^{1/\deg(s_l)}\right\}\right)^{ah}.
\] 
Moreover, $R^{(ah)}$ is generated by $R_{ah}$ over $R_0$.
Thus, as in Claim~\ref{claim:thm:base:strictly:small:sec:1},
by Lemma~\ref{lem:asym:R:asym:M} and Lemma~\ref{lem:R:finite:over:R:m}, we have the assertion.
\CQED

\begin{Claim}
\label{claim:thm:base:free:noetherian:lambda:2}
We may assume that $R_1$ is base point free, that is, $R_1 \otimes \OO_X \to L$ is surjective.
\end{Claim}

Let $\mathcal{I}$  be the ideal sheaf of $X$ given by
\[
\Image(R_1 \otimes \OO_{X} \to L) = \mathcal{I} \cdot L.
\]
Let $\mu : X' \to X$ be the blowing-up with respect to $\mathcal{I}$.
Then $\mathcal{I} \cdot \OO_{X'}$ is invertible. Let $t$ be the canonical section of
of $(\mathcal{I} \cdot \OO_{X'})^{-1}$, that is, $\OO_{X'}(-\zeros(t)) = \mathcal{I} \cdot \OO_{X'}$, and let
$L' = \mathcal{I} \cdot \mu^*(L)$.
Then, as $\left\langle (R_1)^n \right\rangle_{R_0} = R_n$, for $s \in R_n$, 
\[
\tilde{s} := \mu^*(s) \otimes t^{-n} \in H^0(X', nL').
\]
It is easy to see the following properties:
\[
\begin{cases}
\widetilde{s_1 + s_2} = \widetilde{s_1} + \widetilde{s_2}, \ \widetilde{as} = a \tilde{s} & 
(s_1, s_2, s \in R_n, \ a \in \ZZ), \\
\widetilde{s_1 \cdot s_2} = \widetilde{s_1} \cdot \widetilde{s_2} &
(s_1 \in R_n, \ s_2 \in R_{n'}).
\end{cases}
\]
Let $\beta_n : R_n \to H^0(X', nL')$ be the homomorphism given by $\beta_n(s) = \tilde{s}$, and
$R'_n = \beta_n(R_n)$. Then, by the above properties,
\[
\bigoplus_{n=0}^{\infty} \beta_n : \bigoplus_{n=0}^{\infty} R_n \to
\bigoplus_{n=0}^{\infty} R'_n
\]
yields a ring isomorphism.
Let $\Nm'_n$ be the norm of $(R'_n)_{\RR}$ given by $\Vert\beta_n(s)\Vert'_n = \Vert s \Vert_n$
for $s \in (R_n)_{\RR}$. Then
\begin{multline*}
\Vert \beta_n(s) \beta_{n'}(s') \Vert'_{n+n'} = \Vert \beta_{n+n'} (s s') \Vert'_{n+n'} 
= \Vert s s' \Vert_{n+n'} \\
\leq \Vert s \Vert_n \Vert s' \Vert_{n'} =  \Vert \beta_n(s) \Vert'_n \Vert \beta_{n'}(s') \Vert'_{n'}
\end{multline*}
for all $s \in (R_n)_{\RR}$ and $s' \in (R_{n'})_{\RR}$. Thus
$\bigoplus_{n=0}^{\infty}\beta$ extends to a ring isometry
\[
\bigoplus_{n=0}^{\infty} (R_n,\Nm_n) \overset{\sim}{\longrightarrow}
\bigoplus_{n=0}^{\infty} (R'_n,\Nm'_n)
\]
as normed graded rings over $\ZZ$.
Note that $R'_1 \otimes \OO_{X'} \to L'$ is surjective.
Hence the claim follows.
\CQED

\begin{Claim}
\label{claim:thm:base:free:noetherian:lambda:3}
We may assume that $L$ is very ample and
$R_n = H^0(X, nL)$ for $n \gg 1$.
\end{Claim}

By Claim~\ref{claim:thm:base:free:noetherian:lambda:2}, 
\[
\left(\bigoplus_{n=0}^{\infty} R_n\right) \otimes \OO_X \to \bigoplus_{n=0}^{\infty} nL
\]
is surjective, which gives rise to a morphism 
\[
\phi : X \to Z := \Proj\left( \bigoplus_{n=0}^{\infty} R_n\right)
\]
such that $\phi^*(\OO_Z(1)) = L$. 
Note that $Z$ is a projective arithmetic variety.
Moreover, there is a natural injective homomorphism
$\alpha_n : R_n \to H^0(Z, \OO_{Z}(n))$
such that $\phi_n^*(\alpha_n(s)) = s$ for all $s \in R_n$,
where $\phi_n^*$ is the natural homomorphism 
$H^0(Z, \OO_{Z}(n))) \to  H^0(X, nL)$.
If we set $R''_n = \alpha_n(R_n)$, then $\bigoplus_{n=0}^{\infty} \alpha_n$ yields to a ring isomorphism
\[
\bigoplus_{n=0}^{\infty} R_n \overset{\sim}{\longrightarrow}
\bigoplus_{n=0}^{\infty} R''_n,
\]
so that, as in Claim~\ref{claim:thm:base:free:noetherian:lambda:2},
there are norms 
$\Nm''_0, \ldots, \Nm''_n, \ldots$
of
$R''_0, \ldots, R''_n, \ldots$ such that 
\[
\bigoplus_{n=0}^{\infty} (R_n,\Nm_n) \overset{\sim}{\longrightarrow}
\bigoplus_{n=0}^{\infty} (R''_n,\Nm''_n)
\]
as normed graded rings over $\ZZ$.
Moreover, if we set $s''_i = \alpha_1(s_i)$, then $\phi^*(s''_i) = s_i$.
Therefore, $\Bs_{\QQ}(s''_1, \ldots, s''_l) = \emptyset$ on $Z_{\QQ}$.
Further, it is well known that $\alpha_n$ is an isomorphism for $n \gg 1$
(cf. \cite[the proof of Theorem~5.19 and Remark~5.19.2 in Chapter~II]{Hartshorne}). Hence the claim follows.
\CQED

Gathering the assertions of
Claim~\ref{claim:thm:base:free:noetherian:lambda:1} and Claim~\ref{claim:thm:base:free:noetherian:lambda:3}, to prove the corollary,
we may assume the following:
\begin{enumerate}
\renewcommand{\labelenumi}{\rom{(\alph{enumi})}}
\item
$s_1, \ldots, s_l \in R_1$ and
$\Bs_{\QQ}(s_1, \ldots, s_l) = \emptyset$.

\item
$L$ is very ample.

\item
$R_n = H^0(X, nL)$ for $n \gg 1$.
\end{enumerate}
Let $\upsilon : \Sigma_X \to \RR_{> 0}$ be the constant map given by
\[
\upsilon(Y) = \max \{ \Vert s_1 \Vert_1, \ldots, \Vert s_l \Vert_1 \}
\]
for $Y \in \Sigma_X$.
Then $(R,\Nm)$ and $\upsilon$ satisfy the conditions
(1), (2) and (3) of Theorem~\ref{thm:base:strictly:small:sec}.
Hence the corollary follows.
\end{proof}

\begin{Corollary}
\label{cor:base:point:free:small:sec}
Let $\overline{L}$ be a continuous hermitian invertible sheaf on $X$.
If there are a positive integer $n_0$ and
$s_1, \ldots, s_l \in H^0(X, n_0L)$ such that $\Bs_{\QQ}(s_1, \ldots, s_l) = \emptyset$,
then there is $B \in \RR_{>0}$ such that
\[
\lambda_{\QQ}(H^0(X,nL), \Nm_{\sup}) \leq B n^{d(d-1)/2} 
\left( \max \{ \Vert s_1 \Vert_{\sup}, \ldots, \Vert s_l \Vert_{\sup} \} \right)^{n/n_0}
\]
for all $n \geq 1$.
\end{Corollary}

\begin{proof}
By Theorem~\ref{thm:base:free:noetherian:lambda}, it is sufficient to show the following lemma.
\end{proof}

\begin{Lemma}
Let $X$ be a projective variety over a field $k$ and $L$ an invertible sheaf on $X$.
If there is a positive integer $m$ such that $mL$ is base point free,
then $R = \bigoplus_{n=0}^{\infty} H^0(X, nL)$ is noetherian.
\end{Lemma}

\begin{proof}
Since $mL$ is base point free,
there are a projective variety $Z$, an ample invertible sheaf $A$ on $Z$ and
a morphism $\phi : X \to Z$ such that $\phi^*(A) = mL$.
As $A$ is ample, it is well known that
if $F$ is a coherent sheaf on $Z$, then
$R' = \bigoplus_{l=0}^{\infty} H^0(Z, lA)$ is noetherian and 
$\bigoplus_{l=0}^{\infty} H^0(Z, lA \otimes F)$
is a finitely generated $R'$-module. 
Note that
\begin{align*}
R & = \bigoplus_{n=0}^{\infty} H^0(X, nL) = 
\bigoplus_{r=0}^{m-1} \left( \bigoplus_{l=0}^{\infty}H^0(X,  (lm + r)L) \right) \\
& =
\bigoplus_{r=0}^{m-1} \left( \bigoplus_{l=0}^{\infty}H^0(Z,  l A \otimes \phi_*(rL)) \right).
\end{align*}
Therefore $R$ is noetherian because 
$R$ is a finitely generated $R'$-module.
\end{proof}

\begin{Remark}
Theorem~\ref{thm:intro:A} and Corollary~\ref{cor:intro:B} in the introduction are consequences of
Theorem~\ref{thm:base:free:noetherian:lambda} and Corollary~\ref{cor:base:point:free:small:sec} respectively together with
Lemma~\ref{lem:lambda:lambda:prime}.
The following examples show that base point freeness by strictly small sections is
substantially crucial.
\end{Remark}

\begin{Example}
\label{example:projective:line}
Let $\PP^1_{\ZZ} = \Proj(\ZZ[X,Y])$ be the projective line over $\ZZ$ and $\OO(1)$
the tautological invertible sheaf  on $\PP^1_{\ZZ}$.
Then $H^0(\PP^1_{\ZZ}, \OO(d))$ is naturally identified with
$\ZZ[X, Y]_d$.
Let $\beta, \gamma \in (0,1) (= \{ x \in \RR \mid 0 < x < 1 \})$ and
$\alpha := \beta^{1-(1/\gamma)} > 1$.
For each $d \geq 0$, we give a continuous metric $|\cdot |_d$ of $\OO(d)$ as follows:
for $(x:y) \in \PP^1_{\ZZ}(\CC)$ and $s \in \CC[X,Y]_d$,
\[
| s |_{d}(x:y) = \frac{| s(x,y) |}{\left( \max \{ \alpha | x |, \beta | y | \} \right)^d}.
\]
We set $\overline{\OO}(d) = (\OO(d), | \cdot |_d)$. Note that $\overline{\OO}(d) = \overline{\OO}(1)^{\otimes d}$.
Here we have the following:
\begin{align}
\addtocounter{Claim}{1}
\label{eqn:example:projective:line:1}
\left\langle \{ s \in H^0(X, \OO(d)) \mid \Vert s \Vert_{\sup} < 1 \} \right\rangle_{\ZZ} & =
\bigoplus_{ d \gamma < i \leq d} \ZZ X^i Y^{d-i}, \\
\addtocounter{Claim}{1}
\label{eqn:example:projective:line:2}
\left\langle \{ s \in H^0(X, \OO(d)) \mid \Vert s \Vert_{\sup} \leq 1 \} \right\rangle_{\ZZ} & =
\bigoplus_{ d \gamma \leq i \leq d} \ZZ X^i Y^{d-i}.
\end{align}
\begin{proof}
Indeed, by a straightforward calculation, 
\[
\Vert X^i Y^{d-i} \Vert_{\sup} = \frac{1}{\alpha^i\beta^{d-i}}
\]
for $0 \leq i \leq d$. Thus
$X^iY^{d-i}$ is a strictly small section for $i$ with $d \gamma < i \leq d$ because
$\alpha^i\beta^{d-i} > 1$. On the other hand, for $s = \sum_{i=0}^d a_i X^i Y^{d-i} \in \CC[X,Y]_d$,
we can see
\begin{align*}
\Vert s \Vert_{\sup} & \geq \sup\left\{ | s |_d(z:1) \mid | z | = \frac{\beta}{\alpha} \right\} =
\frac{1}{\beta^d} \sup \left\{ | s(z,1) | \mid | z | = \frac{\beta}{\alpha}  \right\} \\
& \geq \frac{1}{\beta^d} \sqrt{ \int_0^{1} \left| s\left(\left(\frac{\beta}{\alpha}\right)e^{2\pi\sqrt{-1}\theta},1\right) \right|^2 d\theta } \\
& =  \frac{1}{\beta^d} \sqrt{\sum_{0 \leq i,j \leq d} \int_0^{1} a_i \bar{a}_j 
\left(\frac{\beta}{\alpha}\right)^{i+j} e^{2\pi\sqrt{-1}(i-j)\theta}d\theta} \\
& = \sqrt{\sum_{i=0}^d \left( \frac{| a_i |}{\alpha^i\beta^{d-i}} \right)^2}.
\end{align*}
Thus, if 
$s = \sum_{i=0}^d a_i X^i Y^{d-i} \in \ZZ[X,Y]_d$
is a strictly small section, then
$a_j = 0$ for $j$ with $0 \leq j \leq d\gamma$ because $\alpha^j\beta^{d-j} \leq 1$.
These observations yield \eqref{eqn:example:projective:line:1}.
Similarly we obtain \eqref{eqn:example:projective:line:2}.
\end{proof}
\end{Example}

\begin{Example}
\label{example:projective:plane}
Let $\PP^2_{\ZZ} = \Proj(\ZZ[X,Y,Z])$ be the projective plane over $\ZZ$ and $\OO(1)$
the tautological invertible sheaf  on $\PP^2_{\ZZ}$.
Let $\Delta$ be the arithmetic subvariety of $\PP^2_{\ZZ}$ given by
the homogeneous ideal $Y \ZZ[X,Y,Z] + Z \ZZ[X,Y,Z]$.
Let $\mu : X \to \PP^2_{\ZZ}$ be the blowing-up along $\Delta$ and $E$ the exceptional divisor of $\mu$.
Note that $E$ is a Cartier divisor.
We set $L = \mu^*(\OO(1)) + \OO_X(E)$ and $R = \bigoplus_{n=0}^{\infty} H^0(X, nL)$.
Since $\mu_*(nL) = \OO(n)$ for all $n \in \ZZ_{\geq 0}$,
the natural ring homomorphism
\[
\mu^*:  \bigoplus_{n=0}^{\infty} H^0(\PP^2_{\ZZ}, \OO(n)) \longrightarrow R
\]
yields a ring isomorphism, and 
\[
\{ x \in X \mid \text{$s(x) = 0$ for all $s \in H^0(X, nL)$}\} = E
\]
for $n \in \ZZ_{>0}$. Here we give a metric $\vert\cdot\vert_{FS}$ of $\OO(1)$
in the following way:
for $s \in H^0(\PP^2_{\CC}, \OO(1)) = \CC[X,Y,Z]_1$ and $(x:y:z) \in \PP^2(\CC)$,
\[
\vert s \vert_{FS} (x:y:z) = 
\frac{\vert s (x,y,z) \vert}{\sqrt{\vert x \vert^2 + \vert y \vert^2 + \vert z \vert^2}}.
\]
We set $\overline{\OO}(n) = (\overline{\OO}(1), \vert\cdot\vert_{FS})^{\otimes n}$.
Then it is easy to check that
$\Vert X^i Y^j Z^k \Vert_{\sup} \leq 1$
for all $n > 0$ and $i,j,k \in \ZZ_{\geq 0}$ with $i+j+k= n$.
Let $t$ be the canonical section of $\OO_X(E)$.
We choose a $C^{\infty}$-metric $\vert\cdot\vert_E$ of $\OO_X(E)$ such that
$\Vert t \Vert_{\sup} < 1$, and set
\[
\overline{L} = \mu^*(\overline{\OO}(1)) + (\OO_X(E), \vert\cdot\vert_E).
\]
Then 
$\Vert \mu^*(X^i Y^j Z^k) \otimes t \Vert _{\sup} < 1$
for all $n > 0$ and $i,j,k \in \ZZ_{\geq 0}$ with $i+j+k= n$.
As a consequence, $R_n$ has  non-empty base loci, but possesses a free basis consisting of strictly small sections.
However, in this example, the free basis comes from the base point free $\ZZ$-module $H^0(\PP^2_{\ZZ}, \OO(n))$.
\end{Example}

\section{Variants of arithmetic Nakai-Moishezon's criterion}

Let $X$ be a projective arithmetic variety and $Y$ an arithmetic subvariety of $X$.
Let $\overline{L}$ be a continuous hermitian invertible sheaf on $X$.
We denote 
\[
\Image(H^0(X, L) \to H^0(Y, \rest{L}{Y}))
\]
by $H^0(X|Y, L)$.
Let $\Vert\cdot\Vert_{\sup,\quot}^{X|Y}$ be the the quotient norm of  $H^0(X|Y, L) \otimes_{\ZZ} \RR$ induced by
\[
H^0(X, L) \otimes_{\ZZ} \RR \twoheadrightarrow H^0(X|Y, L) \otimes_{\ZZ} \RR
\]
and
the norm $\Vert\cdot\Vert_{\sup}$ on $H^0(X, L) \otimes_{\ZZ} \RR$.
As in \cite{MoArLin}, we define $\avol_{\quot}(X|Y, \overline{L})$ to be
\[
\avol_{\quot}(X|Y, \overline{L}) := \limsup_{m\to\infty} \frac{\log \#
\left\{ s \in H^0(X|Y, mL) \mid \Vert s \Vert_{\sup,\quot}^{X|Y} \leq 1 \right\} }{m^{\dim Y}/(\dim Y)!}.
\]
Then we have the following variants of
arithmetic Nakai-Moishezon's criterion.
Theorem~\ref{thm:Nakai:Moishezon:2} is a slight generalization of
the original criterion due to Zhang \cite{ZhPos}, that is,
we do not assume that $L_{\QQ}$ is ample.

\begin{Theorem}
\label{thm:Nakai:Moishezon:1}
If $\avol_{\quot}(X|Y, \overline{L}) > 0$ for all arithmetic subvarieties $Y$ of $X$,
then $L_{\QQ}$ is ample and
there is a positive integer $n_0$ such that, for all $n \geq n_0$,
$H^0(X, nL)$ has a free $\ZZ$-basis consisting of strictly small sections.
\end{Theorem}

\begin{proof}
First of all, note that
\[
\avol\left(Y, \rest{\overline{L}}{Y}\right) \geq \avol_{\quot}(X|Y, \overline{L}) > 0
\]
for all arithmetic subvarieties $Y$ of $X$. In particular, 
by \cite[Corollary~2.4]{Yuan} or \cite[Theorem~4.6]{MoCont}, $\rest{L_{\QQ}}{Y_{\QQ}}$ is big.
Thus, by algebraic Nakai-Moishezon's criterion, $L_{\QQ}$ is ample.
Let us consider a normed graded ring
\[
(R,\Nm) = \bigoplus_{n \in \ZZ_{\geq 0}} (H^0(X, nL), \Nm_{\sup}).
\]
As $L_{\QQ}$ is ample,  $R$ satisfies
the conditions (1) and (2) of Theorem~\ref{thm:base:strictly:small:sec}.
Moreover, if we take a sufficiently small positive number $\epsilon$,
then 
\[
\avol_{\quot}(X|Y, \overline{L} - \overline{\OO}(\epsilon)) > 0
\]
by \cite[(2) in Proposition~6.1]{MoArLin}, 
which means that we can choose a map
$\upsilon : \Sigma_X \to \RR_{> 0}$ such that
$\upsilon(Y) < 1$ for all $Y \in \Sigma_X$ and
the condition (3) of Theorem~\ref{thm:base:strictly:small:sec} holds
for $(R,\Nm)$ and $\upsilon$.
Thus the last assertion follows.
\end{proof}

\begin{Theorem}
\label{thm:Nakai:Moishezon:2}
We assume that $X$ is generically smooth,
the metric of $\overline{L}$ is $C^{\infty}$,
$L$ is nef on every fiber of $X \to \Spec(\ZZ)$ and that
the first Chern form $c_1(\overline{L})$ is
semipositive on $X(\CC)$.
If $\adeg\left((\rest{\hat{c}_1(\overline{L})}{Y})^{\dim Y}\right) > 0$
for all arithmetic subvarieties $Y$ of $X$, 
then $L_{\QQ}$ is ample and
there is a positive integer $n_0$ such that, for all $n \geq n_0$,
$H^0(X, nL)$ has a free $\ZZ$-basis consisting of strictly small sections.
\end{Theorem}

\begin{proof}
By virtue of the Generalized Hodge index theorem \cite[Theorem~6.2]{MoCont},
\[
\avol\left(Y, \rest{\overline{L}}{Y}\right) \geq \adeg\left((\rest{\hat{c}_1(\overline{L})}{Y})^{\dim Y}\right) > 0.
\]
Thus $L_{\QQ}$ is ample as in the proof of Theorem~\ref{thm:Nakai:Moishezon:1}.
In particular, if we set 
\[
(R, \Nm) = \bigoplus_{n \geq 0} (H^0(X, nL), \Nm_{\sup}),
\]
then $R$ satisfies the conditions (1) and (2) of Theorem~\ref{thm:base:strictly:small:sec}.
As $\avol\left(Y, \rest{\overline{L}}{Y}\right) > 0$, we can find a non-zero strictly
small section $s$ of $\rest{n_1L}{Y}$ for some positive integer $n_1$.
By \cite[Theorem~3.3 and Theorem~3.5]{ZhPos},
there are a positive integer $n_2$ and $s' \in H^0(X, n_2n_1L) \otimes \RR$
with $\rest{s'}{Y} = s^{\otimes n_2}$ and $\Vert s' \Vert_{\sup} < 1$.
Thus a map $\upsilon : \Sigma_X \to \RR_{> 0}$ given by
\[
\upsilon(Y) = \left( \Vert s' \Vert_{\sup}\right)^{1/n_1n_2}
\]
satisfies the condition (3) of
Theorem~\ref{thm:base:strictly:small:sec}.
Hence the theorem follows from Theorem~\ref{thm:base:strictly:small:sec}.
\end{proof}

\end{document}